\documentclass[12pt]{amsart}

\usepackage{graphicx}
\usepackage[capitalise]{cleveref}
\usepackage{amsmath,amssymb}
\usepackage{fullpage}
\usepackage{easylist}
\usepackage{color}
\usepackage{enumitem}
\setlist[enumerate]{leftmargin=.5in}
\setlist[itemize]{leftmargin=.5in}
\usepackage[colorinlistoftodos,bordercolor=orange,backgroundcolor=orange!20,linecolor=orange,textsize=scriptsize]{todonotes}

\def\T{\mathsf{T}}
\def\K{\mathsf{K}}
\def\Z{\mathbb{Z}}
\def\R{\mathbb{R}}
\def\M{\mathcal{M}}
\def\barT{\overline{\T}}
\def\barsigma{\overline{\sigma}}
\def\a{\boldsymbol{a}}
\def\b{\boldsymbol{b}}
\def\p{\boldsymbol{p}}
\def\x{\boldsymbol{x}}
\def\y{\boldsymbol{y}}

\def\alt{\operatorname{alt}}
\def\sd{\operatorname{sd}}
\def\S{\mathcal{S}}

\def\ind{\operatorname{ind}}
\def\Xind{\operatorname{Xind}}
\def\Hom{\operatorname{Hom}}
\def\sign{\operatorname{sign}}

\Crefname{thm}{Theorem}{Theorems}

\Crefname{cor}{Corollary}{Corollaries}

\newtheorem{thm}{Theorem}[section]
\newtheorem*{classic-thm}{Theorem}
\newtheorem{lemma}[thm]{Lemma}
\newtheorem{prop}[thm]{Proposition}
\newtheorem{cor}[thm]{Corollary}
\theoremstyle{definition}

\newtheorem{example}{Example}[section]
\newtheorem{remark}{Remark}[section]

\newtheoremstyle{named}{}{}{\itshape}{}{\bfseries}{.}{.5em}{\thmnote{#3}#1}
\theoremstyle{named}
\newtheorem*{namedtheorem}{}

\begin{document}

\title{Multilabeled versions of Sperner's and Fan's lemmas and applications}

\author{Fr\'ed\'eric Meunier and Francis Edward Su}

\address{Department of Mathematics, Harvey Mudd College \\ 301 Platt Blvd\\ Claremont CA 91711, USA}
\email{su@math.hmc.edu}

\address{Universit\'e Paris-Est, CERMICS (ENPC)\\ 6-8 avenue Blaise Pascal\\ 77455 Marne-la-Vall\'ee, France}

\email{frederic.meunier@enpc.fr}
\subjclass[2010]{Primary 55M20; Secondary 54H25, 05E45, 91B32}

\begin{abstract}
We propose a general technique related to the polytopal Sperner lemma for proving old and new multilabeled versions of Sperner's lemma.  A notable application of this technique yields a cake-cutting theorem where the number of players and the number of pieces can be independently chosen. We also prove multilabeled versions of Fan's lemma, a combinatorial analogue of the Borsuk-Ulam theorem, and exhibit applications to fair division and graph coloring. 
\end{abstract}

\keywords{Cake-cutting, consensus-halving, Fan's lemma, graph coloring, labelings, Sperner's lemma}

\maketitle

\section{Introduction}




Sperner's lemma and Fan's lemma are classical results of combinatorial topology.
Notably, they provide elementary constructive proofs of the topological theorems of Brouwer and Borsuk-Ulam and their applications.  The typical object involved with these lemmas is a triangulation of a pseudomanifold with an integer labeling of the vertices. Then, a constraint on the labeling imposes the existence of a simplex with a certain pattern.  
For Sperner's lemma, such a simplex is the discrete analogue of a Brouwer fixed point, and for Fan's lemma, such a simplex is the discrete analogue of the antipodal point whose existence is asserted by the Borsuk-Ulam theorem. Further information about extensions and various applications of these two lemmas is available in the recent survey by De Loera et al.~\cite{de2017discrete}.

These lemmas can be extended to results involving many labelings, in which one seeks the existence of a simplex on which the labelings together exhibit a special pattern.  The oldest results of this type are probably Bapat's ``quantitative'' generalization of Sperner's lemma~\cite{bapat1989constructive} (see \cref{sec:bapat}) and Lee-Shih's generalization of Fan's formula~\cite{lee1998counting}.  


In this paper, we prove multilabeled extensions of Sperner's and Fan's lemmas and discuss various applications. In particular, we show how these extensions can be used to get nice new results in social choice, game theory, graph theory, and topology.




One of the contributions of this paper is a new technique for proving multilabeled versions of Sperner's lemma.  Interestingly, it relies on another generalization of the Sperner lemma, namely the polytopal Sperner lemma~\cite{de2002polytopal}. 
We are not only able to recover several known results using this method, but also to find new ones that seem to be of interest. One of these new results is a theorem about envy-free cake-cutting. 
The original envy-free cake-cutting theorem, found independently by Stromquist~\cite{stromquist1980cut} and by Woodall~\cite{woodall1980dividing} in 1980, 
considers dividing a one-dimensional cake among $k$ players and
ensures under mild conditions on player preferences that there always exists a division of the cake into $k$ connected pieces and an envy-free assignment of these pieces to the $k$ players, i.e., an assignment such that each player prefers the piece she is assigned. 
Woodall actually proved that
such a division exists without knowing the preferences of a ``secretive player'': it is possible to divide the cake into $k$ connected pieces so that whatever choice is made by the ``secretive player'', there is an envy-free assignment of the remaining pieces to the $k-1$ other players: they do not envy each other nor the secretive player.
(This theorem has recently been rediscovered, with a much simpler proof by Asada et al.~\cite{asada2017fair} who were not aware of Woodall's result.)
A special case of our theorem is the following dual version, which is new in the cake-cutting literature.

\begin{cor}
\label{cor:cake}
For any instance of the cake-cutting problem with $k$ players, there exists a division of the cake into $k-1$ connected pieces so that no matter which player leaves, there is an envy-free assignment of the pieces to the remaining $k-1$ players.
\end{cor}

This result may be stated in another way that appeals to those who are familiar with the reality TV show {\it Survivor}, in which the survival skills of a group of people are tested and one person gets ``voted off'' the island each evening.  In that case the theorem above can be restated: 
{\em You can divide a cake into $k-1$ pieces today so that no matter which player gets voted off the island tonight, there will be an envy-free assignment of the pieces to the remaining $k-1$ players.}

We show that Fan's lemma admits multilabeled generalizations too. This is another contribution of this paper. These generalizations extend known applications of Fan's lemma to graph coloring and ``continuous necklace-splitting'' problems. An illustration of these applications is the following multicoloring version of a theorem by Simonyi and Tardos~\cite[Theorem 1]{SiTa06}. The quantity $\ind(\Hom(K_2,G))$ used in the statement is a classical topological invariant associated to a graph $G$; see \cref{subsec:graph} for more details.

\begin{cor}\label{cor:color}
In any collection of $m$ proper colorings of a graph $G$, there is a vertex adjacent to at least $\left\lfloor\frac 1 {2m} \ind(\Hom(K_2,G))\right\rfloor+1$ distinct colors in each of the colorings.
\end{cor}

The multilabeled versions of Sperner's and Fan's lemmas, as well as their applications, are presented and proved in~\cref{sec:sperner,sec:fan}. There is next a brief discussion about a possible ``quantitative'' multilabeled version of Fan's lemma in the same vein as Bapat's theorem for Sperner's lemma (\cref{sec:bapat}). We end the paper with \cref{sec:open}, which gathers the open questions met along this work.

\section{Multilabeled versions of Sperner's lemma}
\label{sec:sperner}

\subsection{Results on multiple labelings and cake division}
Given a triangulation $\T$ of the standard $(n-1)$-dimensional simplex $\Delta^{n-1}=\langle v_1,\ldots,v_n \rangle$, a labeling of its vertices $V(\T)\rightarrow[n]$ is a {\em Sperner labeling} if each vertex $v$ of $\T$ is labeled by an integer $j$ such that $v_j$ is a vertex of the minimal face of $\Delta^{n-1}$ containing $v$.
Sperner's celebrated lemma is the following statement.

\begin{namedtheorem}[Sperner's lemma]
Any triangulation of $\Delta^{n-1}$ with a Sperner labeling has an $(n-1)$-dimensional simplex whose vertices get distinct labels.
\end{namedtheorem}


Our main results regarding multilabeled versions of Sperner's lemma are the following two theorems. 

\begin{thm}[Multilabeled Sperner lemma]
\label{thm:multisperner}
Let $\T$ be a triangulation of $\Delta^{n-1}$ and let $\lambda_1,\ldots, \lambda_m$ be Sperner labelings on $\T$.
\begin{enumerate}[label=\textup{(\arabic*)}]
\item\label{item:manylabels} 
For any choice of positive integers $k_1,\ldots,k_m$ such that $k_1+\cdots+k_m = m+n-1$, 
there exists a simplex $\sigma \in \T$ on which, for each $i$, the labeling $\lambda_i$ uses at least $k_i$ distinct labels.
\item\label{item:manylabelings} 
For any choice of positive integers $\ell_1,\ldots,\ell_n$ such that $\ell_1+\cdots+\ell_n = m+n-1$, 
there exists a simplex $\tau \in \T$ on which, for each $j$, the label $j$ is used in at least $\ell_j$ labelings.
\end{enumerate}
\end{thm}

\begin{example}\label{ex:multisperner}
An illustration of the possible patterns asserted by the theorem is given by \cref{twospaces-sperner}. Let $n=3$ (with labels $1,2,3$) and $m=2$ (with labelings $\lambda_1, \lambda_2$). For $k_1+k_2 = 2 + 2 = 4$, \cref{thm:multisperner} \cref{item:manylabels} asserts the existence of a simplex $\sigma \in \T$ like this one, in which the first labeling uses $k_1=2$ labels ($1$ and $3$) and the second labeling uses $k_2=2$ labels ($2$ and $3$).  This simplex $\sigma$ also exhibits an instance asserted by \cref{thm:multisperner}, \cref{item:manylabelings} for $\ell_1+\ell_2 + \ell_3 = 1 + 1 + 2 = 4$, since the label $1$ appears in $1$ labeling, label $2$ appears in $1$ labeling, and label $3$ appears in $2$ labelings.
\end{example}

The proof of \cref{thm:multisperner} will actually show that \cref{item:manylabels} holds with an additional property: each label is used by at least one of the $\lambda_i$ on the simplex $\sigma$.
When $m=1$, both~\cref{item:manylabels,item:manylabelings} reduce to the usual Sperner's lemma.
Other choices of parameters (e.g., $\lambda_1=\cdots=\lambda_m$ with $\ell_1,\ldots,\ell_n\geq 1$) also yield the usual Sperner's lemma.


\begin{figure}[h]
\begin{center}
\includegraphics[height=1in]{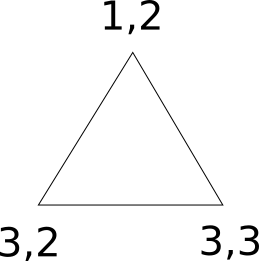}
\caption{\label{twospaces-sperner}Labeled simplex used in \cref{ex:multisperner} to illustrate \cref{thm:multisperner}}
\end{center}
\end{figure}

Originally conjectured by the first author, \cref{item:manylabels} has been proved by Babson~\cite{babson2012meunier} and an elementary and constructive proof has recently been found by Frick et al.~\cite{frick2017achieving}. The proof we propose is new but shares some common points with the second proof. In the next section, we will provide yet another proof of this result, with an approach that clearly departs from the other ones. A kind of ``dual'' statement of \cref{item:manylabels} is present in Section 6 of the cited paper by Frick et al.  However, \cref{item:manylabelings} seems to be the natural ``dual'' version, and we get it with almost the same proof as for \cref{item:manylabels}.



The second theorem is formulated as a ``cake-cutting'' result. Even if  we actually prove Sperner-type results implying it in a standard way, we felt that the cake-cutting formulation makes the statement of the theorem more appealing.

The traditional setting of the cake-cutting problem is the following.  We are given a cake to divide among players. Since the cake will be cut with parallel knives, we can identify the cake with the segment $[0,1]$ so that knife cuts are just points of this segment, and a division is just a partition of the cake into intervals (the fact that a boundary point belongs or not to a given interval does not matter).  The players have preferences satisfying the following two assumptions.  In any division of the cake, each player prefers at least one piece of positive length---\emph{the ``hungry'' assumption}---and may prefer several pieces. And, if a player prefers a particular piece in each division of a converging sequence of divisions, then she also prefers this piece for the limit division---\emph{the ``closed preferences'' assumption}.  An assignment of pieces to players is {\em envy-free} if each player prefers the piece she is assigned to all other pieces.

\begin{thm}
\label{thm:cake}
Consider an instance of the cake-cutting problem with $k$ players. The following holds for all integers $1\leq p,q\leq k$.
\begin{enumerate}[label=\textup{(\arabic*)}]
\item\label{item:secretive} For any subset of $p$ players, there is a division of the cake into $k$ pieces such that, 
no matter which $\lceil \frac{k-p}p\rceil$ pieces are 
removed,
there is an envy-free assignment of $p$ of the remaining pieces to the $p$ players
(the $p$ players do not envy each other nor the pieces that were removed).
\item\label{item:leave} There is a division of the cake into $q$ pieces such that, no matter which $\lceil \frac{k-q}q\rceil$ players leave, there is an envy-free assignment of the pieces to some $q$ of the remaining players.
\end{enumerate}
\end{thm}

When $p=q=k$, we get the usual envy-free result due to Stromquist~\cite{stromquist1980cut} and Woodall~\cite{woodall1980dividing}. When $p=k-1$, we get the ``secretive player'' extension of Woodall \cite{woodall1980dividing} and Asada et al.~\cite{asada2017fair} mentioned in the introduction. When $p=1$, we get that it is always possible to divide a cake into an arbitrary number of pieces a player is indifferent between. The case $q=1$ says something obvious. And \cref{cor:cake} is the special case $q=k-1$, which appears to be a new result in the cake-cutting literature.

One way to think about this theorem is that it asserts: (1) that for given number of players, divisions of cake are possible in which there are enough ties for most-preferred piece so that even if some pieces are removed, you can still get an envy-free assignment among the players, and (2) that for a given number of pieces, divisions of cake are possible in which there are enough pieces chosen by multiple players so that even if some players leave, the pieces can still be assigned in an envy-free way.

Two reviewers asked whether the bounds $\lceil \frac{k-p}p\rceil$ and $\lceil \frac{k-q}q\rceil$ of~\cref{thm:cake} could be improved. For \cref{item:secretive}, the question is whether there are $k$ and $p$ such that the conclusion holds for any instance with a number of removed pieces greater than $\lceil \frac{k-p}p\rceil$.  The question is similar for \cref{item:leave}. We conjecture that the answer is `no', and thus that somehow our theorem is optimal. It is `no'
for $p=q=2$ and $k=4$, as we explain now. Some other cases could probably be settled following a similar method, but a general answer remains to be provided.

Consider first the case of~\cref{item:secretive}.  Suppose player $1$ values the whole cake interval uniformly and player $2$ only values the first half of the cake $[0,1/2]$ uniformly (she does not value any portion of the second half).  Suppose for a contradiction that there is a division of the cake into $4$ pieces $A, B, C, D$ in that order such that, no matter which two pieces are removed, there is an envy-free assignment of the two remaining pieces to the two players.  This must mean that each of the two players is indifferent between at least three preferred pieces in that division (because otherwise they might envy the two removed pieces).  Because of player $1$'s preferences, the largest $3$ of the $4$ pieces must have the same length $1/4+ c$ for some $0 \leq c \leq 1/12$ and the remaining piece has length $1/4-3c$.  This implies that one of the largest $3$ pieces must be completely contained in the first half.  But then player $2$ must value that piece strictly more than both $C$ or $D$ since those pieces contain some of the second half of the cake.  This contradicts the fact that player $2$ had a 3-way tie for her preferred piece.  So the desired division is not possible.


Consider now the case of~\cref{item:leave}. This case is even simpler.  Choose the preferences such that every player has an \emph{indifference point} $x$ in $[0,1]$: a point for which the player is indifferent between choosing $[0,x]$ and $[x,1]$, and such that 
if the cake is cut at any point $y\neq x$, then that player must prefer the piece that contains $x$.  Consider 4 players with distinct indifference points: $x_{1} < x_{2} < x_{3} < x_{4}$.  For such players there is no hope of a division of the cake into two pieces so that no matter which two players leave, there is an envy-free assignment of the pieces to the remaining two players.  Because any point $y$ has a side which contains at most 2 indifference points, if those players leave, the other two players cannot be satisfied.


\subsection{Other applications}

In this section, we present applications of our theorems, or rather, applications of the proof techniques.

The first application concerns rental harmony. Roommates rent together an apartment and they have to decide who should get which room and for which part of the total rent. A division of the rent between the rooms is {\em envy-free} if it is possible to assign a different room to each roommate such that no roommate wishes to exchange her room with another roommate. The following two assumptions are customary, and are the counterparts of the two assumptions made for cake-cutting. A roommate will not envy the others if she gets any room for free. If a roommate prefers a room for every division of a converging sequence of divisions, then she also prefers that room for the limit division. Under these assumptions, there exists an envy-free rent division~\cite{su1999rental}. It is actually quite straightforward to see that a rental harmony version of \cref{thm:cake} holds. We do not state the version we could get in its full generality and we consider only the most interesting cases $p=k-1$ and $q=k-1$. The first case is actually that of the secretive roommates version considered in the paper by Frick et al. \cite{frick2017achieving}. The second case is the following {\em Survivor} rental harmony theorem, which is a new result in this area.

\begin{thm}\label{thm:rental-survivor}
Consider an instance of the rental-harmony problem with $k$ roommates and $k-1$ rooms. Then there exists a division of the rent so that no matter which roommate is kicked out, there is an envy-free assignment of the rooms to the remaining $k-1$ roommates.
\end{thm}

Another application has a microeconomic flavor. We are given $m$ workers and $n$ factories belonging to a same company. The company can choose for each factory the wage of a worker (a unique wage for each factory). The choice of the wages is identified with a vector 
$\x\in \R_{\geq 0}^n$. Each worker $i$ has a continuous utility function 
$u_i\colon[n]\times \R_{\geq 0}^n\rightarrow \R_{\geq 0} $ 
that provides the utility for him to work at factory $j$ with a wage $x_j$ when the other wages are given by the $x_{j'}$ for $j'\neq j$. We suppose that for all $i$ and $j$, we have $u_i(j,\x)=0$ if $ x_j=0$. (The workers do not want to work for free.) The company has a target number of workers assigned to each factory $j$, which is a nonnegative integer denoted by $k_j$ and a total budget $B > 0$ for the wages. We have thus $\sum_{j=1}^nk_j=m$ and $\sum_{j=1}^nk_jx_j=B$. The company would like to choose the wages and to assign the workers to the factories so that no worker will strictly prefer to be assigned to another factory (``make the workers happy'').

\begin{thm}\label{thm:workers}
There exists a choice of wages such that the company is able to assign exactly $k_j$ workers to each factory $j$, while making all of the workers happy.
\end{thm}

%

\subsection{Proofs}

\subsubsection{Preliminaries}\label{subsec:prel}

All multilabeled versions of Sperner's lemma considered in this paper involve a triangulation $\T$ of the standard $(n-1)$-dimensional simplex $\Delta^{n-1}=\langle v_1,\ldots,v_n\rangle$ with $m$ Sperner labelings. Our proofs are based on the following construction. Consider the standard $(m-1)$-dimensional simplex $\Delta^{m-1}=\langle u_1,\ldots,u_m \rangle$ and the polytope  
$$P=\Delta^{m-1}\times\Delta^{n-1}$$
whose vertices are $\{(u_i,v_j)\colon i\in[m], j\in[n]\}$. 
Choose any triangulation $\barT$ of $P$ that refines the product decomposition 
$\Delta^{m-1}\times\T$
without adding new vertices;
this will ensure that any simplex $\barsigma$ in $\barT$ projects naturally to a simplex $\sigma$ in $\T$. (Note that the ``bar'' on $\sigma$ is not an operation: several simplices of $\barT$ projects to a same simplex in $\T$. It is only used for the following consistency: simplices in $\barT$ have a bar, simplices in $\T$ do not.)

We are given $m$ Sperner labelings $\lambda_1,\ldots,\lambda_m$ on $\T$, as in \cref{thm:multisperner}. For the proof of \cref{thm:cake}, these labelings are given by the players' preferences.
Then define the map $\lambda\colon V(\barT)\rightarrow V(P)$ by 
$$\lambda(u_i,v)=(u_i,v_{\lambda_i(v)}).$$

\begin{lemma}\label{lem:sperner-labeling}
The map $\lambda$ maps each vertex of $\barT$ to a vertex of the minimal face of $P$ containing it.
\end{lemma}

\begin{proof}
Consider a vertex $(u_i,v)$ of $\barT$, and let $\sigma\times\tau$ be the minimal face of $P$ containing it. The simplex $\sigma$ must be the single vertex $u_i$ itself and $v$ must be contained in the interior of $\tau$. 
Since $\lambda_i$ is a Sperner labeling, $\lambda_i(v)$ must be an integer $j$ such that $v_j$ is a vertex of $\tau$. Thus $\lambda(u_i,v)$ is $(u_i,v_j)$, which is a vertex of $\sigma\times\tau$, as desired.
\end{proof}

The lemma above actually says that $\lambda$ is a kind of `polytopal' Sperner labeling on $\barT$, using vertices of $P$ as labels.  There is a polytopal Sperner lemma \cite{de2002polytopal}, which asserts the existence of a number of simplices of $\barT$ whose vertices get different labels.  The result crucially depends on the following lemma, which appears in \cite[Proposition 3]{de2002polytopal} and whose underlying topological intuition is well-known: when a ball is mapped to itself so that the restriction to the boundary has degree 1 (as a map of spheres), then the interior of the ball must be covered as well.  


\begin{lemma}\label{lem:surj}
Any continuous map $f$ from a polytope to itself, which satisfies $f(F)\subseteq F$ for all faces $F$ of the polytope, is surjective.
\end{lemma}

The map $\lambda$, which has been defined for each vertex of $\barT$, can be affinely extended on each simplex of $\barT$ to produce a piecewise affine map $\lambda\colon P\rightarrow P$. It is a continuous map and \cref{lem:sperner-labeling} implies that it satisfies the conditions of \cref{lem:surj}. 

Each of our results is then obtained by choosing a special point $\p$ in $P$ and by analyzing the nature of a simplex $\barsigma$ in $\barT$ such that $\p\in\lambda(\barsigma)$. 
Choose $\barsigma$ of minimal dimension; then each vertex of $\barsigma$ is mapped to a distinct vertex of $P$.  

Let $G(\barsigma)$ denote the bipartite graph with vertex bipartition $\{u_1,\ldots,u_m\},\{v_1,\ldots,v_n\}$, and with edges $u_iv_j$ such that $(u_i,v_j)$ is a vertex of $\lambda(\barsigma)$.
By construction, $G(\barsigma)$ is a simple bipartite graph with at most $m+n-1$ edges. 

The point $\p$, as a weighted convex combination of the vertices $(u_i,v_j)$, corresponds in a natural way to a set of weights $w_{e}(\p) \geq 0$ on the edges $e = u_{i} v_{j}$ of $G(\barsigma)$.
Writing $\p=(\a,\b)\in \Delta^{m-1}\times\Delta^{n-1}$ and expressing $\a=(a_{i})$ and $\b=(b_{j})$ in barycentric coordinates, we can see that 
\begin{equation}\label{eq:G}
\sum_{e\in\delta_{G(\barsigma)}(u_i)}w_{e}(\p)=a_i
\qquad\mbox{and}\qquad
\sum_{e\in\delta_{G(\barsigma)}(v_j)}w_{e}(\p)=b_j,
\end{equation}
where the notation $\delta_{G}(x)$ denotes the edges of a graph $G$ incident to a vertex $x$.

\subsubsection{Proof of \cref{thm:multisperner}}

With the following lemma, the proof of \cref{thm:multisperner} is almost immediate.

\begin{lemma}\label{lem:bip_deg}
Consider a bipartite graph $G=(X,Y,E)$ with non-negative weights $\omega_e$ on its edges $e\in E$ and with positive integer weights $s_x$ on its vertices $x\in X$ such that 
$\sum_{x\in X} s_x \geq |E|$.
If for all $x\in X$ and $y\in Y$ we have 
$$\sum_{e\in\delta_G(x)}\omega_e > \frac{s_x-1}{|Y|}
\qquad\mbox{and}\qquad 
\sum_{e\in\delta_G(y)}\omega_e=\frac 1 {|Y|},$$
then the degree of each vertex in $X$ is exactly $s_x$.
\end{lemma}


\begin{proof}
The non-negativity of $\omega_e$ implies that $\omega_e\leq 1/|Y|$ for all edges $e$ of $G$. 
Then the first inequality above implies that $\frac{\deg(x)}{|Y|} > \frac{s_x-1}{|Y|}$ for all $x\in X$. 
Since the $s_x$ are integers, we see $\deg(x)\geq s_x$ for all $x\in X$.  But since the sum of the degrees $\deg(x)$ is 
equal to $|E|$, it is at most the sum of the weights $s_{x}$, and
we have $\deg(x)=s_{x}$ for all $x\in X$.
\end{proof}

\begin{proof}[Proof of \cref{thm:multisperner}, \cref{item:manylabels}]
We choose $\p=(\a,\b)$ such that $a_i = \frac{k_i - 1}{n} + \frac{1}{mn}$ and $b_j=\frac 1 n$.
According to \cref{lem:bip_deg} with $G=G(\barsigma)$, $X=\{u_1,\ldots,u_m\}$, $Y=\{v_1,\ldots,v_n\}$, $\omega_e=w_{e}(\p)$, and $s_{u_i}=k_i$, the degree in $G(\barsigma)$ of each $u_i$ is $k_i$.
This means that $\lambda_i$ assigns at least $k_i$ distinct labels on the vertices of $\sigma$, the projection of $\barsigma \in \barT$ to a simplex of $\T$.
\end{proof}

\begin{proof}[Proof of \cref{thm:multisperner}, \cref{item:manylabelings}]
We choose $\p=(\a,\b)$ such that $a_i =\frac 1 m$  and $b_j=\frac{\ell_j - 1}{m} + \frac{1}{mn}$.
According to \cref{lem:bip_deg} with $G=G(\barsigma)$, $X=\{v_1,\ldots,v_n\}$, $Y=\{u_1,\ldots,u_m\}$, $\omega_e=w_{e}(\p)$, and $s_{v_j}=\ell_j$, the degree in $G(\barsigma)$ of each $v_j$ is $\ell_j$.
This means that  $j$ is used by at least $\ell_j$ labelings on the vertices of $\sigma$, 
the projection of $\barsigma \in \barT$ to a simplex of $\T$.
\end{proof}

We remark that if one desires an algorithmic way to find $\barsigma$, the path-following proof of the polytopal Sperner lemma \cite{de2002polytopal} can be adapted for this purpose.

\subsubsection{Proof of \cref{thm:cake}}

To prove \cref{thm:cake}, we proceed with a technique introduced by the second author~\cite{su1999rental} and identify each point $\y=(y_1,\ldots,y_n)$ in $\Delta^{n-1}$ with a division of the cake into $n$ pieces: numbering the pieces from left to right, $y_j$ is the length of the $j$-th piece. To fit into the proof scheme described in \cref{subsec:prel}, we define $\lambda_i(v)$, for each player $i$ and each vertex $v$ of $\T$, to be the number (in $[n]$) of a preferred piece in the division encoded by $v$.  The ``hungry'' assumption implies that the $\lambda_i$ are Sperner labelings of $\T$.

Instead of \cref{lem:bip_deg}, we use the following lemma.

\begin{lemma}\label{lem:bip_match}
Consider a bipartite graph $G=(X,Y,E)$ with non-negative weights $\omega_e$ on its edges $e\in E$.
If for all $x\in X$ and $y\in Y$ we have 
$$\sum_{e\in\delta_G(x)}\omega_e = \frac 1 {|X|}
\qquad\mbox{and}\qquad
\sum_{e\in\delta_G(y)}\omega_e = \frac 1 {|Y|},$$
then for any subset $Y'\subseteq Y$ of size $\left\lceil\frac{|Y|-|X|}{|X|}\right\rceil$, there exists a matching covering $X$ and missing $Y'$.
\end{lemma}


\begin{proof}
We apply Hall's marriage theorem. Let $X'$ be any non-empty subset of $X$. 
Notice that
the weight on the edges incident to $X'$ is not larger than the weight of the edges incident to $N(X')$.  This, with the hypotheses on the weights, implies $\frac{|X'|}{|X|}\leq\frac{|N(X')|}{|Y|}$. 
Then $\left(1+\frac{|Y|-|X|}{|X|}\right)|X'|\leq |N(X')|$ and since $X'$ is non-empty, we obtain 
$$|X'|+\left\lceil\frac{|Y|-|X|}{|X|}\right\rceil \leq |N(X')|.$$ 
Thus in a graph $H$ obtained from $G$ by removing any subset $Y'\subseteq Y$ with 
$\left\lceil\frac{|Y|-|X|}{|X|}\right\rceil$ vertices, we have the inequality $|X'|\leq |N_H(X')|$ that holds for all non-empty subsets $X'$ of $X$.
\end{proof}

To prove \cref{item:secretive,item:leave}, we choose 
$\p=(\a,\b)$ such that $a_i=\frac 1 m$ and $b_j=\frac 1 n$, but with different meanings of $m$ and $n$ in each instance.
\begin{proof}[Proof of \cref{thm:cake}, \cref{item:secretive}]
Set $m=p$ and $n=k$. 
We apply \cref{lem:bip_match} with $G=G(\barsigma)$, $X=\{u_1,\ldots,u_m\}$, $Y=\{v_1,\ldots,v_n\}$, and $\omega_e=w_{e}(\p)$, 
the weight on an edge of the bipartite graph $G(\barsigma)$ that depends on $\p$.
No matter what $\lceil \frac{n-m}m\rceil$ vertices we remove from $\{v_1,\ldots,v_n\}$, 
the graph $G(\barsigma)$ has a matching covering the vertices in $\{u_1,\ldots,u_m\}$. 
This matching assigns the $m$ players to distinct pieces. 
An edge $u_iv_j$ in $G(\barsigma)$ means that player $i$ is happy with piece $j$ in the division encoded by one of the vertices of the projection $\sigma$ of $\barsigma$ in $\T$. We have thus an envy-free assignment but we are not done yet since $G(\barsigma)$ is built from several divisions.

The triangulation $\T$ being arbitrary, we can assume that $\barsigma$ is such that the diameter of $\sigma$ is arbitrarily small. Since there are only a finite number of possible simple bipartite graphs on $X$ and $Y$, we can make this diameter goes to $0$ while having always the same $G(\barsigma)$. By the closedness of the preferences, at the limit point, the matching assigning the $m$ players to distinct pieces corresponds to an envy-free assignment of a single division.
\end{proof}

\begin{proof}[Proof of \cref{thm:cake}, \cref{item:leave}]
Set $m=k$ and $n=q$. We apply \cref{lem:bip_match} with $G=G(\barsigma)$, $X=\{v_1,\ldots,v_n\}$, $Y=\{u_1,\ldots,u_m\}$, and $\omega_e=w_{e}(\p)$,
the weight on an edge of the bipartite graph $G(\barsigma)$ that depends on $\p$.
No matter what $\lceil \frac{m-n}n\rceil$ vertices we remove from $\{u_1,\ldots,u_m\}$, 
the graph $G(\barsigma)$ has a matching covering the vertices in $\{v_1,\ldots,v_n\}$. 
This matching assigns the $n$ pieces to distinct players in an envy-free way. 
We conclude as in \cref{item:secretive}.
\end{proof}

\subsubsection{Sketch of proof of \cref{thm:rental-survivor}}

The proof consists in applying the relabeling technique proposed by Frick et al.~for the proof in Section 4 of \cite{frick2017achieving} and then in following the steps of the proof of \cref{thm:cake}, \cref{item:leave}. 

\subsubsection{Proof of \cref{thm:workers}}

The proof is almost the same as for \cref{thm:cake}, \cref{item:leave}.

Consider the map $\phi\colon\Delta^{n-1}\rightarrow \R_{\geq 0} ^n$ defined by 
$$\phi(\y)=\left(\frac {By_1} {k_1},\ldots,\frac {By_n} {k_n}\right).$$ It bijectively maps the points of $\Delta^{n-1}$ with choices of wages satisfying the budget constraint.
To fit into the proof scheme described in \cref{subsec:prel}, we define $\lambda_i(v)$, for each player $i$ and each vertex $v$ of $\T$, to be the index $j$ of a factory maximizing $u_i(j,\phi(\y))$, with $\y$ being the coordinates of $v$. The assumption on $u_i$ makes $\lambda_i$ a Sperner labeling of $\T$.

Choose $\p=(\a,\b)$ such that $a_i=\frac 1 m$ and $b_j=\frac{k_j} m$. Since $\sum_{j=1}^nk_j=m$,  the point $\p$ is indeed in $\Delta^{m-1}\times\Delta^{n-1}$, as required by  \cref{subsec:prel}. Set $G=G(\barsigma)$, $X=\{v_1,\ldots,v_n\}$, $Y=\{u_1,\ldots,u_m\}$, and $\omega_e=w_{e}(\p)$. Now, instead of using \cref{lem:bip_match}, we apply directly Hall's marriage theorem. Because of \cref{eq:G}, we have $\sum_{e\in\delta_G(u_i)}\omega_e=\frac 1 m$ for all $i$ and  $\sum_{e\in\delta_G(v_j)}\omega_e=\frac{k_j} m$ for all $j$. For any subset $X'\subseteq X$, we have $$\sum_{j\colon v_j\in X'}k_j=m\sum_{j\colon v_j\in X'}\sum_{e\in\delta_G(v_j)}\omega_e\leq m\sum_{i\colon u_i\in N(X')}\sum_{e\in\delta_G(u_i)}\omega_e=|N(X')|.$$ Hall's marriage theorem (actually, an obvious weighted extension) implies then that there exists a subset $F\subseteq E$ covering exactly $k_j$ times each vertex $v_j\in X$ (i.e., each factory) and exactly once each vertex in $Y$ (i.e., each worker). We conclude as in the proof of \cref{thm:cake}, \cref{item:secretive}. It is in this last step that the continuity of the $u_i$'s matters.

\begin{remark}\label{rk:gale}
While 
all theorems of that section can also be obtained with the averaging technique introduced by Gale~\cite{gale1984equilibrium} for proving his ``permutation'' generalization of the KKM lemma---the technique used in the papers Asada et al.~\cite{asada2017fair} and Frick et al.~\cite{frick2017achieving}---our approach makes clear the symmetry between the labelings and the labels.
\end{remark}




\section{Multilabeled versions of Fan's lemma}\label{sec:fan}

We now establish multilabeled versions of Fan's lemma in direct analogy to the two parts of the multilabeled Sperner lemma.



\subsection{Fan's lemma and $\Z_2$-index}
A {\em free simplicial $\Z_2$-complex} is a simplicial complex on which there is a free $\Z_{2}$-action.  A {\em Fan labeling} of a free simplicial $\Z_2$-complex is a labeling of its vertices with non-zero integers such that (i) no adjacent vertices have labels that sum to zero (adjacency condition), and (ii) the two vertices of any orbit have labels that sum to zero (antisymmetry condition). 

In such a labeling, a simplex is {\em alternating} (with respect to the labeling) if the signs alternate when the vertices are ordered according to the absolute value of their labels. The \emph{sign} of an alternating simplex (either positive or negative) is the sign of its first (lowest) label by absolute value.

The following lemma is due to Fan~\cite{Fa56}. It deals with a centrally symmetric triangulation of a $d$-dimensional sphere $\S^{d}$, which is a free simplicial $\Z_2$-complex via the antipodal map. The antisymmetry condition for a Fan labeling requires then that labels at antipodal vertices sum to zero.

\begin{namedtheorem}[Fan's lemma]
In any centrally symmetric triangulation of $\S^{d}$ with a Fan labeling, there is an alternating $d$-dimensional simplex.
\end{namedtheorem}

Tucker's lemma~\cite{tucker1945some} is equivalent to the special case: 
{\em In a centrally symmetric triangulation of $\S^{d}$ with a Fan labeling, there is at least one vertex with a label whose absolute value is larger than $d$.}

It has recently been realized~\cite[Proposition 1]{AlHoMe17} that Fan's lemma remains true when the $d$-sphere $\S^d$ is replaced by any free simplicial $\Z_2$-complex with $\Z_2$-index equal to $d$. 
We remind the reader that the {\em $\Z_2$-index} of a free simplicial $\Z_2$-complex $\K$, denoted $\ind(\K)$, is the minimal dimension of the sphere to which there is a continuous map from $\K$ commuting with the $\Z_2$-action.  
If $\K$ is a triangulated sphere, the $\Z_2$-action is the antipodal map, and the Borsuk-Ulam theorem is the equality $\ind(\S^d)=d$. 

\begin{namedtheorem}[Fan's lemma for a simplicial $\Z_2$-complex]
In any free simplicial $\Z_2$-complex $\K$ with a Fan labeling, there is an alternating $\ind(\K)$-dimensional simplex.
\end{namedtheorem}

In fact, by symmetry, there is a positive alternating simplex as well as a negative alternating simplex.



\subsection{Coincidences of alternating simplices}
The next theorem can be seen as the Fan-type generalization of \cref{thm:multisperner}, \cref{item:manylabels}.  Using the derivation of Sperner's lemma from Fan's lemma (see~\cite{nyman2013borsuk, vzivaljevic2010oriented}), this theorem, when $\K$ is a centrally symmetric triangulation of the $(n-1)$-dimensional sphere, actually provides yet another proof of \cref{thm:multisperner}, \cref{item:manylabels}.  
Fan's lemma is the special case $m=1$ on a triangulated $d$-sphere (whose index is $d$).

\begin{thm}
\label{thm:multifan}
Let $\lambda_1,\ldots,\lambda_m$ be $m$ Fan labelings of a free simplicial $\Z_2$-complex $\K$.
For any choice of non-negative integers $d_1,\ldots,d_m$ summing to $\ind(\K)$, 
there exists a simplex $\sigma$
in $\K$ that for each $i$ 
has a $d_i$-dimensional alternating face with respect to $\lambda_i$.
\end{thm}

\begin{example}\label{ex:multifan}
An illustration of the possible patterns asserted by the theorem is given by \cref{twospaces-fan}. 
On a triangulated $2$-sphere with $m=2$ labelings $\lambda_1, \lambda_2$, suppose $d_1+d_2 = 1 + 1 =  \ind(\K)=2$.  Then \cref{thm:multifan} asserts the existence of a simplex $\sigma \in \T$ like this one, in which there is a $1$-dimensional alternating face with respect to $\lambda_1$ (the left edge with labels $+1,-7$) and a $1$-dimensional alternating face with respect to $\lambda_2$ (the right edge with labels $+1,-3$).
\end{example}


\begin{figure}[h]
\begin{center}
\includegraphics[height=1in]{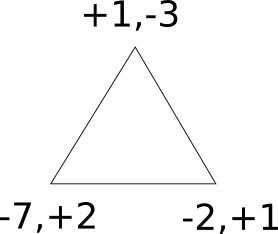}
\caption{\label{twospaces-fan}Labeled simplex used in \cref{ex:multifan} to illustrate \cref{thm:multifan}}
\end{center}
\end{figure}

Alishahi~\cite[Lemma 17]{alishahi2016colorful} introduced a new technique for proving Fan's lemma and its generalizations for other group actions. This technique consists in defining a labeling on the first barycentric subdivision of the centrally symmetric triangulation $\T$---the labels recording the length of the alternation on the simplices of $\T$---and then in using Tucker's lemma to show the existence of alternating simplices of certain dimensions. Our approach for proving \cref{thm:multifan} relies on this idea, and we will again use this technique in the proof of \cref{thm:multifan-dual-weaker} below. As noticed by Alishahi (personal communication), this technique can actually be used to get an extension of \cref{thm:multifan} where $\Z_2$ is replaced by any $\Z_q$.

We shall use the following notation.  Let $\alt_{\lambda_i}(\sigma)$ be the number of vertices of a maximal alternating face of $\sigma$ with respect to the labeling $\lambda_i$. Let $s_{\lambda_i}(\sigma)$ be the sign of the alternation of such a face, i.e., the sign of the label in the alternation with the smallest absolute value. It is straightforward to check that $s_{\lambda_i}(\sigma)$ is well-defined when $\lambda_i$ is a Fan labeling.  Given a simplicial complex $\K$, let $\sd(\K)$ denote the \emph{barycentric subdivision} of $\K$: vertices of $\sd(\K)$ are the simplices of $\K$, and the simplices of $\sd(\K)$ are chains of faces $\K$ ordered by inclusion.

\begin{proof}[Proof of \cref{thm:multifan}]
For $\sigma \in \K$, 
let $i^{\star}(\sigma)$ be the smallest labeling index $i$ 
such that $\sigma$ does not have an alternating face of dimension $d_{i}$ according to $\lambda_{i}$ (hence $\alt_{\lambda_i}(\sigma) \leq d_{i}$) 
and if there is no such index $i$, call $\sigma$ \emph{desirable} and let $i^{\star}(\sigma) = m$.  Our goal is to show the existence of desirable simplices.
Define 
$$
\mu(\sigma)=\pm [ d_{1} + \cdots + d_{i^{\star}(\sigma) - 1}  +\alt_{\lambda_{i^{\star}(\sigma)}}(\sigma) ]
$$
where the sign is chosen to be the sign of the alternation in $\lambda_{i^{\star}(\sigma)}$.  
Since $d_{1} + \cdots + d_{i^{\star}(\sigma) - 1}$ is at most $\ind(\K)-d_{m}$ and 
$\alt_{\lambda_{i^{\star}(\sigma)}}(\sigma)$ is at most $\dim\K+1$, we see $\mu$ is a labeling of the vertices of $\sd(\T)$ by non-zero integers being at most
$\ind(\K) - d_{m} + \dim(\K) + 1$ in absolute value.
Note that a simplex $\sigma$ is desirable if and only if $|\mu(\sigma)| > \ind(\K)$.  

We claim $\mu$ is a Fan labeling of $\sd(\K)$.  
Antisymmetry of $\mu$ follows easily from its definition.
For the adjacency condition of $\mu$ on $\sd(\K)$, we consider $\tau$ a face of $\sigma$, and check that $\mu(\tau)$ and $\mu(\sigma)$ cannot sum to $0$.
We have $i^{\star}(\tau) \leq i^{\star}(\sigma)$. If $i^{\star}(\tau) < i^{\star}(\sigma)$, then $|\mu(\tau)| \leq d_{1} + \cdots + d_{i^{\star}(\sigma) - 1} < |\mu(\sigma)|$ (even when $\sigma$ is desirable), hence $\mu(\tau)$ and $\mu(\sigma)$ cannot sum to $0$. Suppose now that $i^{\star}(\tau) = i^{\star}(\sigma)$. Then denoting this value by $i^{\star}$, we see
that $|\mu(\tau)| = |\mu(\sigma)|$ can only happen when 
$\alt_{\lambda_{i^{\star}}}(\tau) =  \alt_{\lambda_{i^{\star}}}(\sigma)$ (even when $\tau$ or $\sigma$ is desirable).
But then the maximal alternating face of $\tau$ is a maximal alternating face of $\sigma$, and the label of lowest absolute values in $\tau$ and $\sigma$ must have the same sign, hence $\mu(\tau)=\mu(\sigma)$ so their sum is not $0$, either.

Then by Fan's lemma for a simplicial $\Z_2$-complex, there is a simplex of dimension $\ind(\K)$ in $\sd(\K)$ specified by a maximal chain of simplices in $\K$:
$$\sigma_0\subseteq\cdots\subseteq\sigma_{\ind(\K)}$$ 
with alternation $$1\leq \mu(\sigma_0)<\cdots<(-1)^{\ind(\K)}\mu(\sigma_{\ind(\K)}).
$$ 
This latter chain of inequalities implies immediately that $|\mu(\sigma_{\ind(\K)})| >\ind(\K)$, which implies that $\sigma_{\ind(K)}$ is desirable.
\end{proof}

The proof above relies on Fan's lemma which is known to have a constructive proof when $\K$ is a triangulated $d$-sphere containing a flag of hemispheres \cite{prescott-su}. In that latter case, we have thus also a constructive method for finding a desirable simplex in our multilabeled extension of the Fan lemma.

We show now two applications of \cref{thm:multifan}. A typical consequence of Fan's lemma in combinatorics is the existence of large colorful bipartite subgraphs in proper colorings of graphs with large topological lower bounds. The first application is a strengthening of this result. The second one is in the context of the Hobby-Rice theorem, also called the ``continuous necklace-splitting theorem'', and consensus-halving.

\subsubsection{Graph application}\label{subsec:graph}

Since the foundational paper of Lov\'asz solving the Kneser conjecture~\cite{Lo79}, there is a whole machinery for designing topological lower bounds on the chromatic number of graphs. The {\em Hom complex} of a graph $G=(V,E)$, denoted $\Hom(K_2,G)$, is the poset whose elements are the pairs $(A,B)$ of non-empty disjoint subsets of $V$ inducing a complete bipartite graph and whose order $\preceq$ is given by: $(A,B)\preceq(A',B')$ if $A\subseteq A'$ and $B\subseteq B'$.  

We shall associate $\Hom(K_2,G)$ with its order complex: a simplicial complex whose vertices are poset elements and simplices are poset chains.
There is a natural free $\Z_2$-action on this complex given by exchanging $A$ and $B$. 

One of the largest topological lower bounds is provided by the $\Z_2$-index of the Hom complex:
$$\chi(G)\geq\ind(\Hom(K_2,G))+2.$$

In a properly colored graph, a {\em colorful} subgraph is a subgraph whose vertices get distinct colors.  
The above inequality was strengthened by Simonyi, Tardif, and Zsb\'an~\cite{SiTaZs13}, who showed that in a properly colored graph $G$, there is in fact a colorful complete bipartite subgraph 
$K_{\lceil d/2\rceil +1,\lfloor d/2\rfloor +1}$ where $d = \ind(\Hom(K_2,G))$.
We can use our \cref{thm:multifan} to extend their result to multiple proper colorings.

\begin{thm}
\label{thm:color}
Consider a graph $G$ colored with $m$ proper colorings $c_1,\ldots,c_m$. For any choice of positive integers $d_1,\ldots,d_m$ summing to $\ind(\Hom(K_2,G))$, there exists a complete bipartite subgraph that for each $i$ contains a colorful $K_{\lceil d_i/2\rceil + 1,\lfloor d_i/2\rfloor + 1}$ with respect to $c_i$.
\end{thm}

\begin{proof}
Let $G=(V,E)$.  Think of each coloring as a map $c_i: V \rightarrow\{1,2,\ldots\}$.
Recall that each vertex of $\Hom(K_2,G)$ is a pair $(A,B)$ of non-empty disjoint subsets of $V$ that induce a complete bipartite subgraph of $G$.  Define 
$$\lambda_i(A,B) = \pm \max_{v\in A\cup B}c_i(v)$$ 
with sign $-$ (resp. $+$) if the maximum is attained in $A$ (resp. $B$). 
Then $\lambda_i$ is a Fan labeling, since antisymmetry is obvious from the definition, and the adjacency condition follows from $c_i$ being a proper coloring: if $(A,B)\preceq(A',B')$ and 
$\max_{v\in A\cup B}c_i(v)$ is achieved in $A$ while $\max_{v\in A'\cup B'}c_i(v)$ is achieved in $B'$, then the two maxima must be different $c_i$ colors because each vertex in $A \subseteq A'$ is adjacent to each vertex in $B'$.

Since $\sum_i d_i=\ind(\Hom(K_2,G))$, \cref{thm:multifan} implies there exists a simplex of $\Hom(K_2,G)$ with alternating $d_i$-dimensional faces, for $i=1,\ldots,m$.  Each such face is a chain
$$(A_0,B_0) \preceq (A_1,B_1) \preceq \cdots \preceq (A_{d_i}, B_{d_i})$$
on which $\lambda_i$ alternates; 
by the definition of $\lambda_i$, there are $v_j \in A_j \cup B_j$ for $j=0,\ldots,d_i$ such that 
$\lambda_i(A_j,B_j)=c_i(v_j)$ and $c_i(v_0) < \cdots < c_i(v_{d_i})$.  

Suppose without loss of generality $v_0 \in B_0$.  Then $a = \max_{v\in A_0} c_i(v)$ is achieved by some $v=\bar v$ and since $c_i$ is proper, the color $a$ is strictly less than $c_i(v_0)$.
This means that $\bar v$ is in the bipartite subgraph 
$(A_{d_i}, B_{d_i})$ but must be distinct from the $d_i + 1$ vertices $v_j$, $j=0,\ldots,d_i$ because they all have larger colors than $\bar v$.  The bipartite subgraph induced by these $d_i+2$ vertices must be colorful according to $c_i$ and of type $K_{\lceil d_i/2\rceil +1,\lfloor d_i/2\rfloor +1}$ as desired.
\end{proof}

Choosing $d_i\simeq\frac{1}{m} \ind(\Hom(K_2,G))$ for all $i$ produces \cref{cor:color} stated in the introduction.

\subsubsection{Consensus-halving}\label{subsec:consensus}

Given a family $\M$ of absolutely continuous measures on $[0,1]$, we define a {\em $t$-splitting} to be a partition of $[0,1]$ into $t+1$ intervals $I_j$ such that 
$$\mu \left(\bigcup_{\text{\textup{ odd $j$}}}I_j\right)=\mu \left(\bigcup_{\text{\textup{ even $j$}}}I_j\right)$$
for all $\mu \in \M$.  
We do not require that the intervals be numbered $1,2,\ldots,n$ when going from left to right. (Note that however, whether or not we add this condition does not matter: if we have such a partition without this condition on the numbering, we have it also with it since either they can be numbered from left to right, or there are two adjacent intervals with indices of the same parity; in the latter case, they can be merged and we can add a new interval of length $0$ and renumber all intervals $1,2,\ldots,n$ going from left to right.)

The Hobby-Rice theorem~\cite{hobby1965moment} asserts that if $\M$ has size $t$, then there is a $t$-splitting of $\M$.  There have been many generalizations of that theorem; see for example~\cite{alon1987splitting,de2008splitting}. 

Here, we propose yet another generalization of this result, based on the our \cref{thm:multifan}. Note that the measures are not necessarily probability measures.



\begin{thm}
\label{thm:multi-halving}
Consider finite collections $\M_1,\ldots,\M_m$ of absolutely continuous measures on $[0,1]$, a positive integer $n$, and positive integers $k_1,\ldots,k_m$ summing to $m+n-1$. Then there exists a partition of $[0,1]$ into $n$ intervals $I_1,\ldots,I_n$ such that, for each $i \in [m]$, one of the properties holds:
\begin{itemize}
\item These intervals provide an $(n-1)$-splitting of $\M_i$.
\item 
The equality $$\mu\left(\bigcup_{\text{\textup{ odd $j$}}}I_j\right)-\mu\left(\bigcup_{\text{\textup{ even $j$}}}I_j\right)=\pm \max_{\mu'\in\M_i}\left|\mu'\left(\bigcup_{\text{\textup{ odd $j$}}}I_j\right)-\mu'\left(\bigcup_{\text{\textup{ even $j$}}}I_j\right)\right|$$
is attained for at least $k_i$ measures $\mu$ in $\M_i$, each sign, $+$ and $-$, being attained at least $\lfloor k_i/2\rfloor$ times.
\end{itemize}
\end{thm}

\Cref{thm:multi-halving} generalizes the Hobby-Rice theorem already for $m=1$: if $\M_1$ has $n-1$ measures, then the second possibility cannot hold since $k_1=n$, and thus there must be an $(n-1)$-splitting. The case $m=1$ is close to a result by P\'alv\H{o}lgyi \cite[Corollary 4.6]{palvolgyi2009combinatorial} that ensures in the discrete case that one can somehow ``control'' the imbalance in an $(n-1)$-splitting when the number of measures is $n$. In P\'alv\H{o}lgyi's result, we can even decide for each measure which part of the partition between odd and even indices is advantaged. It is a generalization of a result by Simonyi~\cite{Si08}, which is the same result without the control on the imbalance.

If we remove the condition on the finiteness of the $\M_i$, the theorem does not hold anymore as shown by the following simple example for $m=1$. For every positive integer $r$, we introduce the function
$$f_r\colon x\in[0,1]\longmapsto\left\{\begin{array}{ll}r^3 & \mbox{if $\displaystyle{x\in\left[\frac 1 {r+1},\frac 1 r \right]}$} \smallskip\\ 0 & \mbox{otherwise,}\end{array}\right.$$ 
and the measure $\nu_r$ whose density is given by $f_r$. Let $\M_1$ be the set of all $\nu_r$. Consider now any partition of $[0,1]$ into $n$ intervals $I_1,\ldots,I_n$. There is an infinite number of $r$ such that $ [\frac 1 {r+1},\frac 1 r]$ is fully contained into one of the $I_j$'s. Thus, the $n$ intervals do not form an $(n-1)$-splitting and we have moreover 
$$\sup_{\mu'\in\M_1}\left|\mu'\left(\bigcup_{\text{\textup{ odd $j$}}}I_j\right)-\mu'\left(\bigcup_{\text{\textup{ even $j$}}}I_j\right)\right|=+\infty$$ since $\lim_{r\rightarrow+\infty}r^3(\frac 1 r - \frac 1 {r+1})=+\infty$. None of the properties ensured by the theorem can hold.

The Hobby-Rice theorem also yields a social science interpretation in terms of {\em consensus halving}~\cite{SiSu03}. If $[0,1]$ represents a cake to be cut and each measure in $\M$ represents a hungry person, then the theorem says there is a \emph{consensus-halving} among $t$ people with just $t$ cuts: for such a cut, there is consensus among all people in $\M$ that the odd-index pieces are the same size as the even-index pieces.  


An interpretation of \cref{thm:multi-halving} in that spirit goes as follows: 
\emph{Consider $m$ finite groups of people and 
non-negative
integers $d_1,\ldots,d_m$ summing to $n-1$. Then given any cake, there exists a division of that cake into $n$ pieces, each assigned to one of two portions $A$ and $B$, such that for each group, 
either: all people in that group believe $A$ and $B$ are exactly the same size; or:
there exists $\gamma_{i} > 0$ such that 
at least 
$\lceil (d_i+1)/2\rceil$
 believe $A$ is larger by exactly $\gamma_{i}$ and at least $\lfloor (d_i+1)/2\rfloor$ believe that $B$ is larger by exactly $\gamma_{i}$ and everyone else is somewhere in between.
}


\begin{proof}[Proof of \cref{thm:multi-halving}]
Let us first assume that no $\M_i$ admits an $(n-1)$-splitting.

We arbitrarily index the measures in each $\M_i$ with positive integers: $\mu_1^i,\mu_2^i,\ldots$  
Let $\T$ be a centrally symmetric triangulation of the $(n-1)$-dimensional unit $L_1$-sphere 
$$\partial\Diamond^n=\left\{(y_1,\ldots,y_n)\in\R^n\colon\sum_{j=1}^{n}|y_j|=1\right\}.$$ 
Let $Y_{0}=0$ and $Y_j=\sum_{j'=1}^j|y_{j'}|$. 
For each vertex $v$ of $\T$ with coordinates $(y_1,\ldots,y_n)$, we define $\lambda_i(v)=\pm a^*$, where $a^*$ is the smallest $a$ for which 
$\left|\sum_{j=0}^{n-1} \sign(y_j)~\mu_{a}^i([Y_j,Y_{j+1}])\right|$ is maximal, 
and the sign of $\lambda_i(v)$ is the sign of the expression in the absolute value.

We check that $\lambda_i$ is a Fan labeling. Antisymmetry is clearly satisfied. By compactness of $\partial\Diamond^n$, finiteness of the $\M_i$'s,  and continuity of the measures, there is some $\delta>0$ so that $$\max_a\left|\sum_{j=0}^{n-1}\sign(y_j)~\mu_{a}^i([Y_j,Y_{j+1}])\right|>\delta$$ for all $i$
and all $(y_1,\ldots,y_n)\in\partial\Diamond^n$. We can choose $\T$ with a sufficiently small mesh size so that $$(y_1,\ldots,y_n)\longmapsto\sum_{j=0}^{n-1}\sign(y_j)~\mu_{a}^i([Y_j,Y_{j+1}])$$ varies by less than $\delta$ on any simplex, and for any $i$ and $a$.

Then, since the dimension of $\partial\Diamond^n$ is $n-1$, according to \cref{thm:multifan}, there is a simplex with an alternating $(k_i-1)$-face with respect to each $\lambda_i$. Making the mesh size of $\T$ go to $0$ and using the compactness of $\partial\Diamond^n$, we get the existence of intervals $J_1,\ldots,J_n$ and a partition $A,B$ of $[n]$ such that, for each $i$, the equality
$$\mu\left(\bigcup_{j\in A}J_j\right)-\mu\left(\bigcup_{j\in B}J_j\right)=\pm \max_{\mu'\in\M_i}\left|\mu'\left(\bigcup_{j\in A}J_j\right)-\mu'\left(\bigcup_{j\in B}J_j\right)\right|$$
is attained $k_i$ times, with the `$+$' sign for at least $\lfloor k_i/2\rfloor$ measures $\mu$ in $\M_i$ and is attained with the `$-$' sign for at least $\lfloor k_i/2\rfloor$ measures $\mu$ in $\M_i$. Now, by merging consecutive intervals whose indices belong to the same subset $A$ or $B$, we get the desired statement, except that we may have less than $n$ intervals; in that case, just add empty intervals to complete the collection.

Now, let us consider the general case. We arbitrarily choose a collection $\mathcal{P}$ of $n$ absolutely continuous measures $\rho_{1},\ldots,\rho_{n}$ on $[0,1]$ with disjoint supports and with total weight equal to some $\varepsilon>0$. We define $\M'_i=\M_i\cup\mathcal{P}$. No $\M_i'$ has an $(n-1)$-splitting. Applying what we already proved on $\M'_1,\ldots,\M'_m$, we get that there are $n$ intervals $I_1,\ldots,I_n$ (depending on $\varepsilon$) such that, for each $i$, one of the two properties holds:
\medskip
\NewList
\begin{easylist}\ListProperties(Numbers1=l, FinalMark1={)})
§ \label{a}$\left|\mu\left(\bigcup_{\text{\textup{ odd $j$}}}I_j\right)-\mu\left(\bigcup_{\text{\textup{ even $j$}}}I_j\right)\right|\leq \varepsilon$ for all $\mu\in\M_i$.
\medskip
§ \label{b}$\mu\left(\bigcup_{\text{\textup{ odd $j$}}}I_j\right)-\mu\left(\bigcup_{\text{\textup{ even $j$}}}I_j\right)=\pm \max_{\mu'\in\M_i}\left|\mu'\left(\bigcup_{\text{\textup{ odd $j$}}}I_j\right)-\mu'\left(\bigcup_{\text{\textup{ even $j$}}}I_j\right)\right|$ is attained for at least $k_i$ measures in $\M_i$, each sign, $+$ and $-$, being attained at least $\lfloor k_i/2\rfloor$ times.
\end{easylist} 
\medskip
Indeed, this follows from noting that for each $i$, either we are in case \ref{a}, or the maximum in not attained for a measure in $\mathcal{P}$.

Consider a sequence of values for $\varepsilon$ converging to $0$. Up to taking a subsequence,  we can assume that each $i$ falls under the same case \ref{a} or the same case \ref{b} for all these values of $\varepsilon$, and, by finiteness of the $\M_i$'s, we can even assume that when it falls under case \ref{b}, the equality is always attained with exactly the same measures and the same signs. By compactness and absolute continuity of the measures, we get the sought conclusion.
\end{proof}

The necklace-splitting theorem ensures that there is a certain rounding property for the Hobby-Rice theorem when one works with ``beads'' instead of measures; see~\cite{alon1986borsuk, goldberg1985bisection}. We do not know whether the rounding property still holds for \cref{thm:multi-halving}.

\subsection{Further multilabeled generalizations of Fan's lemma}

Similarly, the next theorem can be seen as the Fan-type generalization of \cref{thm:multisperner}, \cref{item:manylabelings}. \Cref{cor:multifan-dual} provides an alternative proof of \cref{thm:multisperner}, \cref{item:manylabelings} in a same way \cref{thm:multifan} provides an alternative proof of \cref{thm:multisperner}, \cref{item:manylabels}.



\begin{thm}
\label{thm:multifan-dual-weaker}
Consider a free simplicial $\Z_2$-complex $\K$  with $\ind(\K)=n-1$. Let $\lambda_1,\ldots,\lambda_m$ be $m$ Fan labelings of $\K$ with labels $\{\pm 1,\ldots,\pm N\}$. 
For any choice of positive integers 
$\ell_1,\ldots,\ell_N$ summing to $m+N-1$, 
there exists a simplex $\sigma = \langle v_{1},\ldots,v_{n} \rangle$ in $\K$, 
label numbers $1\leq j_1 \leq \cdots \leq j_n\leq N$, 
and indices $i_1, \ldots , i_n$ in $[m]$
such that 
\begin{enumerate}[label=\textup{(\alph*)}]
  \item\label{v} $\lambda_{i_{k}}(v_{k})=(-1)^{k} j_{k}$,
  \item\label{j} if $j_{k-1}=j_{k}$, then $i_{k-1} < i_{k}$, and
  \item\label{l} for each $k$, labels $+ j_{k}$ or $-j_{k}$ are present on $\sigma$ in at least $\ell_{j_k}$ of the labelings.
\end{enumerate}
\end{thm}

Note here the theorem asserts that there is an ``alternating'' sequence of labels, though the label absolute values $j_{k}$  are not necessarily distinct. 
\Cref{v} guarantees the asserted alternating labels appear on different vertices of $\sigma$.
\Cref{j} shows that if a sequence of asserted alternating labels have identical absolute values, 
then those labels successively appear in a sequence of labelings of increasing index.
When $m=1$, the theorem becomes the usual Fan lemma.



\begin{proof}[Proof of \cref{thm:multifan-dual-weaker}]
For $\sigma\in \K$, let $r_j(\sigma)$ be the number of labelings in which $-j$ or $+j$ is present on $\sigma$.  We define $\mu(\sigma)$ as follows. 
Let $j^{\star}(\sigma)$ be the largest $j$ such that $r_j(\sigma)\geq \ell_j$. 
Such a $j^{\star}(\sigma)$ exists because $\sum_{j=1}^N(\ell_j-1)=m-1$ and $\sum_{j=1}^Nr_j(\sigma)\geq m$ (each labeling contributes at least one unit to this sum). 
Let $i^{\star}(\sigma)$ be the largest $i$ such that $\lambda_{i}(v)= \pm j^{\star}(\sigma)$ for a vertex $v \in \sigma$.

Now set 
$$\mu(\sigma)= \pm [ m \cdot (j^{\star}(\sigma)-1)  + i^{\star}(\sigma) ]$$
where the $\pm$ sign is determined by the sign of the label $\pm j^{\star}(\sigma)$ that appears on $\sigma$ in $\lambda_{i^{\star}(\sigma)}$ 
(both labels cannot appear in $\sigma$ since $\lambda_{i^{\star}(\sigma)}$ is a Fan labeling, hence adjacent vertices in $\sigma$ cannot have labels that sum to zero).


We claim $\mu$ is a Fan labeling on $\sd(\K)$ with labels from $\{\pm 1,\ldots,\pm mN\}$.  Antisymmetry of $\mu$ follows from the symmetry of $j^{\star}$ and $i^{\star}$ and the antisymmetry of $\lambda_{i^{\star}}$ on $\K$.  Also, $\mu$ satisfies the adjacency condition, because if $\tau \subset \sigma$ are adjacent vertices in $\sd(\K)$,
then either $j^{\star}(\tau) < j^{\star}(\sigma)$ in which case $\mu(\tau) + \mu(\sigma) \neq 0$, 
or $j^{\star}(\tau) = j^{\star}(\sigma)$ and $i^{\star}(\tau) < i^{\star}(\sigma)$ in which case $\mu(\tau) + \mu(\sigma) \neq 0$, or $j^{\star}(\tau) = j^{\star}(\sigma)$ and $i^{\star}(\tau) = i^{\star}(\sigma)$ and the adjacency condition on 
$\lambda_{i^{\star}}$ ensures $\mu(\tau)= \mu(\sigma)$.

So $\mu$ is a Fan labeling and according to Fan's lemma, there is an $(n-1)$-dimensional simplex in $\sd(\K)$, which corresponds to a chain of simplices in $\K$ of successive dimension (here $\sigma_{1}$ has dimension $0$):
$$\sigma_1\subseteq\cdots\subseteq\sigma_{n}$$ 
with $$1\leq -\mu(\sigma_1)<\cdots<(-1)^{n}\mu(\sigma_{n})\leq mN.$$ 

We claim that $\sigma=\sigma_n$ is the simplex we are looking for, with label number $j_k$ given by $j^{\star}(\sigma_k)$, and index $i_k$ given by $i^{\star}(\sigma_k)$. Indeed, the fact that $|\mu|$ increases going from $\sigma_{k-1}$ to $\sigma_{k}$ means that $j^{\star}$ and/or $i^{\star}$ also increases by including a vertex $v_{k}$ in the simplex $\sigma_{k}$;
hence this $v_{k}$ must be labeled by $\pm j_k$ in $\lambda_{i_{k}}$, with the same sign as $\mu(\sigma_k)$, i.e., $(-1)^k$. This gives \cref{v,j}. The definition of $j^{\star}(\sigma_k)$ implies~\cref{l}. 
\end{proof}

We say that $m$ Fan labelings $\lambda_1,\ldots,\lambda_m$ are {\em compatible} if $\lambda_i(u)+\lambda_{i'}(u')\neq 0$ for any adjacent vertices $u,u'$ and any pair $i,i'$ (where $u$ is not considered as being adjacent to itself).

\begin{cor}\label{cor:multifan-dual}
Consider a free simplicial $\Z_2$-complex $\K$  with $\ind(\K)=n-1$. 
Let $\lambda_1,\ldots,\lambda_m$ be $m$ compatible Fan labelings of $\K$ with labels $\{\pm 1,\ldots,\pm N\}$. 
For any choice of positive integers $\ell_1,\ldots,\ell_N$ summing to $m+N-1$,
there exists a simplex $\sigma$ and $1\leq j_1< \cdots < j_n \leq N$ such that for each $k$, 
labels $+j_k$ or $-j_k$ are present on $\sigma$ in at least $\ell_{j_k}$ of the labelings.
\end{cor}

\begin{proof}
We apply \cref{thm:multifan-dual-weaker}. 
Because of the additional condition, \cref{v} implies that $j_k<j_{k+1}$ for all indices $k$.
\end{proof}

\begin{remark}
In the proofs of \cref{thm:multifan,thm:multifan-dual-weaker}, the labeling $\mu$ is increasing. This implies that the same proofs actually provide slightly stronger versions of these results, where the $\Z_2$-index $\ind(\cdot)$ is replaced by the always non-smaller cross-index $\Xind(\cdot)$; see~\cite{SiTaZs13} for definition of this latter index and discussion about it. In turn, this implies that the slightly stronger versions of \cref{thm:color,cor:color,cor:multifan-dual} with the cross-index in place of the index are also true. However, while there are simplicial complexes with distinct  $\Z_2$-index and cross-index, it is not known whether this can occur for Hom complexes of graphs; see~\cite{SiTaZs13}.
\end{remark}

Note that in contrast to \cref{thm:multifan}, \cref{cor:multifan-dual} requires the Fan labelings to be compatible.  The following example below shows that without this hypothesis, the corollary is not true.

\begin{example}
\label{bad-example}
Let $\T$ be a symmetric triangulation of the unit circle $\S^{1}$ with vertices at $(\pm 1,0)$ and $(0,\pm 1)$, and sufficiently small mesh size (less than $0.1$ will do).
Let $\lambda_{1}$ be anti-symmetric labeling of the vertices such that 
$\lambda_{1}((1,0))=+2$, 
$\lambda_{1}((-1,0))=-2$, 
and otherwise $\lambda_{1}((x,y))$ is $+1$ if $y>0$ and $-1$ if $y<0$.
Let $\lambda_{2}$ be anti-symmetric labeling of the vertices such that 
$\lambda_{2}((0,1))=+2$, 
$\lambda_{2}((0,-1))=-2$, 
and otherwise $\lambda_{2}((x,y))$ is $+1$ if $x>0$ and $-1$ if $x<0$.

These are Fan labelings $\lambda_{i}\colon V(\T)\rightarrow \{\pm 1, \pm 2 \}$ that satisfy the hypotheses of \cref{cor:multifan-dual} with $n=m=N=2$ except that they are not compatible, since any simplex near $\p=(1/\sqrt{2}, -1/\sqrt{2})$ will have adjacent vertices with labels of $-1$ in $\lambda_{1}$  and $+1$ in $\lambda_{2}$.  
The conclusion of \cref{cor:multifan-dual} now fails for $\ell_{1}=1, \ell_{2} =2$, since 
we must have $j_{1}=1, j_{2}=2$, but the small mesh size of the triangulation guarantees there is no simplex with labels $+2$ or $-2$ in both labelings.
\end{example}

However, 
\Cref{bad-example} does satisfy the conditions of \cref{thm:multifan-dual-weaker} for $n=m=N=2$. We can see how it applies when $\ell_{1}=1, \ell_{2} =2$.  Since $\lambda_{1}$ and $\lambda_{2}$ cannot simultaneously have labels $\pm 2$ for vertices on the same simplex $\sigma$, it must be the case that \cref{thm:multifan-dual-weaker} produces a simplex $\sigma$ with label numbers $j_{1}=j_{2}=1$ and $i_{1}=1, i_{2}=2$.  This there is an alternating sequence of labels $-1, +1$ on $\sigma$ each of which makes the redundant claim that the labels $\pm 1$ appear in a least $\ell_{1}=1$ labelings and at least $\ell_{2}=2$ labelings.
For instance, this will occur for a simplex near $\p$.  

The next result is a multilabeled version of Bacon's lemma~\cite{bacon}, which states that in a Fan labeling of $\S^{n-1}$, with labels taken in $\{\pm 1,\ldots,\pm n\}$, all feasible subsets of $n$ distinct labels occur as labels of a simplex. The proof of this multilabeled version illustrates the fact that Gale's averaging trick (mentioned in \cref{rk:gale}) can also be used for proving Fan-type results. 

\begin{prop}\label{prop:gale-kyfan}
Consider a free simplicial $\Z_2$-complex $\K$ with $\ind(\K)=n-1$. Let $\lambda_1,\ldots,\lambda_{n}$ be $n$ compatible Fan labelings of $\K$ with labels in $\{\pm 1,\ldots,\pm n\}$ with the additional condition that $\lambda_i(u)+\lambda_{i'}(u)\neq 0$ for any pair of indices $i,i'$ and any vertex $u$.
Then, for every $(\alpha_1,\ldots,\alpha_{n})\in\{-1,+1\}^{n}$, there exists a simplex $\sigma$ of $\K$ and a permutation $\pi$ of $[n]$ such that for each $j\in[n]$, the integer $\alpha_j\cdot j$ is a value taken by $\lambda_{\pi(j)}$ on $\sigma$.
\end{prop}

Note that the condition of the proposition is stricter than of \cref{cor:multifan-dual} since it requires that 
among the labels found at a single vertex, no pair of labels sums to zero.

\begin{proof}[Proof of \cref{prop:gale-kyfan}]
Use each $\lambda_i$ to construct a piecewise affine map $L_{i}$ from 
the underlying space $\|\K\|$ to $\partial\Diamond^{n}$ in the following way:
for each vertex $v$ set $L_{i}(v)=\pm  e_{|\lambda_{i}(v)|}$ 
(where $e_{k}$ is the $k$-th basis vector) 
choosing the sign to agree with $\lambda_{i}(v)$,
then extend $L_{i}$ linearly across each simplex. 
%
Using Gale's averaging trick mentioned in \cref{rk:gale}, 
we define $L(\x)=\frac 1{n}\sum_{i=1}^{n}L_i(\x)$ for every $\x\in\|\K\|$. 
Because of the compatibility conditions and $\lambda_i(u)+\lambda_{i'}(u)\neq 0$, 
the map $L$ has still its image in $\partial\Diamond^{n}$. 
Since $L$ is a $\Z_2$-map and since $\K$ has $\Z_2$-index equal to $n-1$, 
the map $L$ is surjective. 
(Otherwise, a $\Z_2$-map from $\K$ into $\S^{n-2}$ would exist by a standard topological argument.) 

Consider $\p=(p_1,\ldots,p_{n})$ defined by $p_j=\frac 1 {n}\alpha_j$. 
Since $L$ is surjective, there exists $\y\in\|\K\|$ such that $L(\y)=\p$. 
Define $z_{ij}$ to be the $j$-th component of $L_i(\y)$; 
note that $z_{ij}$ is positive (resp.~negative) if and only if $+j$ (resp.~$-j$) appears as a $\lambda_{i}$ label in a minimal simplex containing $\y$, because $\lambda_{i}$ is a Fan labeling.
We have $\sum_{i=1}^{n}|z_{ij}|=1$ for all $j$ 
(since $L(\y)=\p$)
and $\sum_{j=1}^{n} |z_{ij}|=1$ for all $i$
(since $L(\y) \in \partial\Diamond^{n}$).

Consider the bipartite graph $G$, 
with on one side the vertices $i=1,\ldots,n$ 
and on the other side the vertices $j=1,\ldots,n$, and with edges the pairs $ij$ such that $z_{ij}\neq 0$. 
For every subset $X$ of $j$-vertices, we have 
$$
|X|=\sum_{i=1}^{n} \sum_{j \in X} |z_{ij}|  \leq  \sum_{i \in N(X)} \sum_{j=1}^{n} |z_{ij}| = |N(X)|,
$$ 
the inequality following from noting the right double sum is a sum over more edges.
Then Hall's marriage theorem ensures that $G$ has a matching covering the $j$-vertices. For any such $j$, we define $\pi(j)$ to be the integer $i$ with which $j$ is matched.

The pair we are looking for is $(\sigma,\pi)$, where $\sigma$ is a simplex of $\K$ containing $\y$. Indeed, we have $z_{\pi(j) j} \neq 0$, which implies that $\lambda_{\pi(j)}(v)=\alpha_j\cdot j$ for at least one vertex of $\sigma$.
\end{proof}

\begin{remark}
We probably could get other generalizations of \cref{prop:gale-kyfan} in the spirit of the ``Sperner''-versions of \cref{thm:cake}, with a similar approach, by playing with the point $\p$ used in the proof. But we were not interested in going further in that direction.
\end{remark}

\section{Bapat's theorem and Lee-Shih's formula}\label{sec:bapat}

As mentioned in the introduction, Bapat's theorem~\cite{bapat1989constructive} is the first multilabeled version of Sperner's lemma. It implies Gale's permutation generalization of the KKM lemma, but it is more general, since it has a quantitative conclusion.

\begin{namedtheorem}[Bapat's theorem]
Let $\T$ be a triangulation of $\Delta^{n-1}$ with $n$ Sperner labelings $\lambda_1,\ldots,\lambda_n$. Consider the pairs $(\sigma,\pi)$, where $\sigma$ is an $(n-1)$-dimensional simplex of $\T$ and $\pi$ is a bijection $V(\sigma)\rightarrow[n]$, such that the $\lambda_{\pi(v)}(v)$ for $v\in V(\sigma)$ are all different. For each such pair, order the vertices of $\sigma$ so that $\lambda_{\pi(v)}(v)$ is increasing along them. Then the difference between the number of such pairs with $\sigma$ positively oriented by this order and the number of such pairs with $\sigma$ negatively oriented is equal to $n!$ in absolute value.
\end{namedtheorem}

While writing this paper, the natural question of whether a similar generalization of Fan's lemma holds arose. We were not able to settle this question. However, for some special triangulations, it is easy to get a statement in that spirit, as we explain now.



Any centrally symmetric triangulation $\T$ of $\S^{n-1}$ provides a triangulation $\T/\Z_2$ of the $(n-1)$-dimensional projective space, obtained by identifying antipodal simplices.

\begin{thm}\label{thm:multi-bapat}
Let $\T$ be a centrally symmetric triangulation of $\S^{n-1}$
with $n$ compatible Fan labelings $\lambda_1,\ldots,\lambda_n$. If the simplicial complex $\T/ \Z_2$ 
is balanced, then there are at least $n!$ pairs $(\sigma,\pi)$ with $\sigma\in\T$ and $\pi$ a bijection $V(\sigma)\rightarrow[n]$, such that 
$$0<-\lambda_{\pi(v_1)}(v_1)<\lambda_{\pi(v_2)}(v_2)<\cdots<(-1)^{n}\lambda_{\pi(v_{n})}(v_{n})$$ 
where $\langle v_1,\ldots,v_{n}\rangle=\sigma$.
\end{thm}

A $d$-dimensional simplicial complex is {\em balanced} if there is a coloring of its vertices with $d+1$ colors such that every $d$-dimensional simplex is colorful, i.e., has its vertices of distinct colors. 

\begin{proof}[Proof of \cref{thm:multi-bapat}]
Denote by $\lambda_1,\ldots,\lambda_{n}$ the $n$ compatible Fan labelings. 
Since we are assuming $\T/ \Z_2$ is balanced, we can
fix an arbitrary coloring $c\colon V(\T)\rightarrow[n]$ such that each $(n-1)$-dimensional simplex is colorful. For each permutation $\pi'$ of $[n]$, the labeling $\lambda^{\pi'}$ defined by $\lambda^{\pi'}(v)=\lambda_{\pi'(c(v))}(v)$ is a Fan labeling and there is a negative alternating simplex according to this labeling. For any pair of distinct permutations $\pi'$, each $(d-1)$-dimensional simplex has at least one vertex where the two values of $\pi'(c(v))$ are different. Each choice of $\pi'$ provides thus a different pair $(\sigma,\pi)$ with the desired property (and with $\pi=\pi'\circ c$).
\end{proof}

We do not know whether the statement still holds if we remove the balancedness condition. Fan proved his lemma by induction on the dimension on the sphere, with the help of a formula relating the number of alternating simplices of a pseudomanifold to the number of alternating simplices with positive sign on its boundary~\cite{Fa56}. A similar formula exists for multilabelings: this is precisely the Lee-Shih formula mentioned in the introduction. Unfortunately, because of issues related to the orientation of the sphere, mimicking Fan's proof with Lee-Shi's formula does not seem to lead to any non-trivial result. 

These issues about orientation may explain why there is an oriented version of Sperner's lemma, due to Brown and Cairns~\cite{brown1961strengthening}, while such an oriented version does not seem to exist for Fan's lemma.

\begin{namedtheorem}[Oriented Sperner's lemma]
Let $\T$ be a triangulation of $\Delta^{n-1}$ with a Sperner labeling. Consider the $(n-1)$-dimensional simplices of $\T$ whose vertices get distinct labels. For each such simplex, order its vertices so that the labels are increasing along them. Then the difference between the number of such simplices that are positively oriented by this order and the number of those that are negatively oriented is equal to $1$ in absolute value.
\end{namedtheorem}

(Note that the oriented Sperner's lemma is the special case of Bapat's theorem where all $\lambda_i$ are equal.) We are not aware of any similar counterpart for Fan's lemma.

\section{Open questions and conjectures}\label{sec:open}

Many open questions and conjectures have arisen while writing this paper. First, there is this question about the ``optimality'' of the bounds in~\cref{thm:cake}, mentioned after its statement in Section~\ref{sec:sperner}.

Other questions stem from considering~\cref{thm:multi-halving}. We mentioned at the end of~\cref{subsec:consensus} that the Hobby-Rice theorem is a special case and admits a ``discrete'' version, namely the necklace-splitting theorem. It remains to decide whether such a discrete version also exists for~\cref{thm:multi-halving} in its full generality. Another question related to \cref{thm:multi-halving} is whether there is a way to show that finiteness of the $\M_i$ is necessary, without using measures of total weight going to infinity as we did right after the statement of the theorem.

Finally, it would be nice to settle the question of the existence of an ``oriented Fan lemma'', as mentioned at the end of~\cref{sec:bapat}.

\subsubsection*{Acknowledgments} We thank the referees for their useful comments, which helped improve the paper. In particular, they have been at the origin of several open questions.

\bibliographystyle{siamplain}
\bibliography{multi-sperner}

\end{document}